\documentclass[letterpaper,reqno,11pt]{amsart}
\usepackage[margin=1.2in]{geometry}

\usepackage{amssymb,amsmath,amsthm, amscd, enumerate, mathrsfs}
\usepackage{graphicx, hhline}
\usepackage[all]{xy}
\usepackage[usenames]{color}
\usepackage{hyperref}
\hypersetup{colorlinks=true}

\theoremstyle{plain}
\newtheorem{thm}{Theorem}[section] 
\newtheorem{cor}[thm]{Corollary}
\newtheorem{prop}[thm]{Proposition}
\newtheorem{conj}[thm]{Conjecture}
\newtheorem{lem}[thm]{Lemma}

\theoremstyle{definition} 
\newtheorem{defn}[thm]{Definition}
\newtheorem{eg}[thm]{Example} 
\theoremstyle{remark}
\newtheorem{rem}[thm]{Remark}
\newtheorem{ques}[thm]{Question}

\newtheorem*{claim}{Claim}
\newtheorem*{acknowledgement}{Acknowledgments}

\newcommand{\sO}{\mathcal{O}}
 
\newcommand{\Z}{\mathbb{Z}}
\newcommand{\N}{\mathbb{N}} 
\newcommand{\Q}{\mathbb{Q}} 
 
\newcommand{\C}{\mathbb{C}}

\newcommand{\pP}{\mathbb{P}}
\newcommand{\ba}{\mathfrak{a}}
\newcommand{\fp}{\mathfrak{p}}
\newcommand{\m}{\mathfrak{m}}

\newcommand{\depth}{\mathop{\mathrm{depth}}\nolimits}

\newcommand{\Spec}{\mathop{\mathrm{Spec}}\nolimits}
\newcommand{\Proj}{\mathop{\mathrm{Proj}}\nolimits}

\newfont{\bg}{cmr17 scaled\magstep5}

\newcommand{\cd}{\mathrm{cd}}

\newcommand{\Pic}{\mathrm{Pic}}

\newcommand{\flo}[2]{\Big\lfloor \frac{#1}{#2} \Big\rfloor}

\DeclareMathOperator{\Cl}{Cl}

\DeclareMathOperator{\height}{ht}
\DeclareMathOperator{\pd}{pd}
\DeclareMathOperator{\reg}{reg}

\title{{On the relationship between depth\\ and cohomological dimension}}

\author{Hailong Dao}
\address{Department of Mathematics, University of Kansas, Lawrence, KS 66045-7523, USA}
\email{hdao@ku.edu}

\author{Shunsuke Takagi}
\address{Graduate School of Mathematical Sciences, University of Tokyo, 3-8-1 Komaba, Meguro-ku, Tokyo 153-8914, Japan}
\email{stakagi@ms.u-tokyo.ac.jp}

\keywords{Local cohomology, Cohomological dimension, Local Picard groups}
\subjclass[2010]{13D45, 14B15, 32C36}

\dedicatory{{Dedicated to Professor~Yuji~Yoshino on the~occasion of his~sixtieth~birthday.}}

\begin{document}

\begin{abstract}
Let $(S, \m)$ be an $n$-dimensional regular local ring essentially of finite type over a field and let $\ba$ be an ideal of $S$. 
We prove that if $\depth S/\ba \ge 3$, then the cohomological dimension $\mathrm{cd}(S, \ba)$ of $\ba$ is less than or equal to $n-3$. 
This settles a conjecture of Varbaro for such an $S$. 
We also show, under the assumption that $S$ has an algebraically closed residue field of characteristic zero, that if $\depth S/\ba \ge 4$, then $\mathrm{cd}(S, \ba) \le n-4$ if and only if the local Picard group of the completion $\widehat{S/\ba}$ is torsion. 
We give a number of applications, including a vanishing result on Lyubeznik's numbers, and sharp bounds on cohomological dimension of ideals whose quotients satisfy good depth conditions such as Serre's conditions $(S_i)$. 
\end{abstract}

\maketitle
\markboth{H.~DAO and S.~TAKAGI}{Depth and cohomological dimension}

%%%%%%%%%%%%%%%%%%%%%%%%%%%%%%%%%%%%%%%%%%%%%%%%%%%%%%%%%%%%%%%%%%%%%%%%%%%%%%%%%%%%%%%%%%%%%%%%%%%%%%%%

\section{Introduction}
{
Local cohomology was introduced by Grothendieck in the early 1960s and has become since a fundamental tool in commutative algebra and algebraic geometry. 
It is important to know when they vanish.
Let $S$ be a Noetherian local ring and $\ba$ be an ideal of $S$. 
Then the \textit{cohomological dimension} $\mathrm{cd}(S, \ba)$ of $\ba$ in $S$ is defined by 
\[
\mathrm{cd}(S, \ba)=\sup\{i \in \Z_{\ge 0} \mid H^i_{\ba}(M) \ne 0 \textup{ for some $S$-module $M$}\}.
\]
This invariant has been studied by many authors including Hartshorne \cite{Ha2}, Ogus \cite{Og}, Peskine-Szpiro \cite{PS}, Huneke-Lyubeznik \cite{HL}, Lyubeznik \cite{Lyu}, Varbaro \cite{Va}, etc.  
It follows from a classical vanishing theorem of Grothendieck that $\mathrm{cd}(S, \ba)$ is less than or equal to the dimension of $S$. 
A natural question to ask is under what conditions one can obtain a better upper bound on $\mathrm{cd}(S, \ba)$.  
In this paper, we assume that $S$ is an $n$-dimensional regular local ring containing a field and investigate the relationship between $\mathrm{cd}(S, \ba)$ and the depth of $S/\ba$. 

The first result of this type is that if $\depth S/\ba \ge 1$, then $\mathrm{cd}(S, \ba) \le n-1$, which is an immediate consequence of the Hartshorne-Lichtenbaum vanishing theorem \cite{Ha2}. 
It also follows from results of Ogus \cite{Og} and Peskine-Szpiro \cite{PS} (see also \cite{HL}) that 
if $\depth S/\ba \ge 2$, then $\mathrm{cd}(S, \ba) \le n-2$. 
In fact, Peskine-Szpiro proved a more general result in positive characteristic: if $S$ is a regular local ring of characteristic $p>0$ and if $\depth S/\ba \ge i$, then $\mathrm{cd}(S, \ba) \le n-i$. 
One might expect that an analogous result holds in equal characteristic zero, but there exists a class of examples where $S$ is a localization of a polynomial ring over a field of characteristic zero, $\depth S/\ba \ge 4$ and $\mathrm{cd}(S, \ba)=n-3$ (see Example \ref{segre}). 
On the other hand, Varbaro \cite{Va} proved that if $\mathfrak{b}$ is a homogeneous ideal of a polynomial ring $T=k[x_1, \dots, x_n]$ over a field $k$ and if $\depth T/\mathfrak{b} \ge 3$, then $\mathrm{cd}(T, \mathfrak{b}) \le n-3$.
He conjectured that a similar statement holds in a local setting. 

\begin{conj}\label{varbaro conj}
Let $S$ be an $n$-dimensional regular local ring containing a field and let $\ba$ be an ideal of $S$. 
If $\depth S/\ba \ge 3$, then $\mathrm{cd}(S, \ba) \le n-3$.
\end{conj}

Motivated by the above conjecture, we consider a necessary and sufficient condition for $\mathrm{cd}(S, \ba)$ to be less than $n-2$ when $S$ is a regular local ring essentially of finite type over a field of characteristic zero. 
The following is our main result.  

\begin{thm}[cf.~Theorem \ref{3rd vanishing in char 0}]\label{main thm}
Let $(S, \m, k)$ be an $n$-dimensional regular local ring essentially of finite type over a field of characteristic zero with residue field $k$ algebraically closed. 
Let $\ba$ be an ideal of $S$ and set $R=S/\ba$. 
Suppose that  $\depth R \ge 2$ and $H^2_{\m}(R)$ is a $k$-vector space. 
Then $\mathrm{cd}(S, \ba) \le n-3$ if and only if the torsion subgroup of $\mathrm{Pic}(\Spec \widehat{R} \setminus \{\m \widehat{R} \})$ is finitely generated, where $\widehat{R}$ is the $\m R$-adic completion of $R$. 
\end{thm}

We note that an analogous statement does not hold in positive characteristic (see Remark \ref{remark in positive char}). 
Also, the assumption on $H^2_{\m}(R)$ cannot be removed (see Example \ref{h2}).
As a corollary of Theorem \ref{main thm}, we give an affirmative answer to Conjecture \ref{varbaro conj} when $S$ is essentially of finite type over a field (Corollary \ref{vanishing}). 

We also study the case where $\depth S/\ba \ge 4$. 
In this case, as we have mentioned above, $\mathrm{cd}(S, \ba)$ is not necessarily less than $n-3$. 
We give a necessary and sufficient condition for $\mathrm{cd}(S, \ba)$ to be less than $n-3$ in a form similar to that of Theorem \ref{main thm}.

\begin{thm}[cf.~Theorem \ref{picard group}]
Let the notation be the same as that used in Theorem \ref{main thm}, and suppose that  $\depth R \ge 4$. 
Then $\mathrm{cd}(S, \ba) \le n-4$ if and only if $\mathrm{Pic}(\Spec \widehat{R} \setminus \{\m \widehat{R} \})$ is torsion. 
\end{thm}

We have several applications of our results.  
We obtain vanishing results of the Lyubeznik numbers $\lambda_{i,j}(R)$, numerical invariants of a Noetherian local ring $R$ containing a field introduced by Lyubeznik \cite{Lyu} (see Definition \ref{Lyubeznik def} for their definition). 
In particular, we prove that if $R$ is a $d$-dimensional local ring essentially of finite type over a field satisfying Serre's condition $(S_3)$, then $\lambda_{d-1, d}(R)=0$ (Proposition \ref{S3}).
We also have an extension of a recent result by Katzman-Lyubeznik-Zhang \cite{KLZ} which in turn extended the classical result of Hartshorne on connectedness of the punctured spectrum (Proposition \ref{KLZex}). 
Finally, we give sharp bounds on cohomological dimension of ideals whose quotients satisfy good depth conditions such as Serre's conditions $(S_i)$. 

\begin{thm}[=Theorem \ref{serre}]
Let $S$ be an $n$-dimensional regular local ring containing a field and let $\ba \subset S$ be an ideal of height $c$. 
\begin{enumerate}
\item If $S/\ba$ satisfies Serre's condition $(S_2)$ and $\dim S/\ba\geq 1$, then 
\[\cd(S,\ba)\leq n-1-\flo{n-2}{c}.\]
\item Suppose that $S$ is essentially of finite type over a field. 
If $S/\ba$ satisfies Serre's condition $(S_3)$ and $\dim S/\ba \ge 2$, then 
\[\cd(S,\ba)\leq n-2-\flo{n-3}{c}.\]
\end{enumerate} 
\end{thm}

Such results partly answer a question raised in \cite[Question 2.5]{Hu}.

\vspace*{0.75em}

In this paper, all rings are Noetherian commutative rings with unity and $\C$ is the field of complex numbers. When discussing the completion of a local ring $(R, \m)$, we will mean the $\m$-adic completion of $R$. 
We denote the completion of $R$ by $\widehat{R}$ and the completion of the strict henselization of the completion of $R$ by $\widehat{(\widehat{R})^{\rm sh}}$. 
}
\begin{small}
\begin{acknowledgement}
The authors would like to thank Matteo Varbaro for inspirations on this work and generously sharing his unpublished notes and many useful comments.  
We are also indebted to Bhargav Bhatt for helpful discussions and pointing out a mistake in a previous version of this work. 
We are grateful to Brian Conrad, Craig Huneke, Atsushi Ito and Yusuke Nakamura for valuable conversions. 
The first author was partially supported by NSF grant  DMS 1104017.
The second author was partially supported by Grant-in-Aid for Scientific Research (C) No.26400039 from JSPS. 
\end{acknowledgement}
\end{small}

%%%%%%%%%%%%%%%%%%%%%%%%%%%%%%%%%%%%%%%%%%%%%%%%%%%%%%%%%%%%%%%%%%%%%%%%%%%%%%%%%%%%%%%%%%%%%%%%%%%%%%%%%%%%%%%%%%%%%%%%%%%%%%%%%%%%%%%%%%%%

\section{Main results on cohomological dimension}\label{vanishing section}

First we recall the notion of local Picard groups. 
\begin{defn}
Let $(R, \m)$ be a Noetherian local ring. 
The \textit{local Picard group} $\mathrm{Pic}^{\rm loc}(R)$ of $R$ is the Picard group $\mathrm{Pic}(\Spec R \setminus \{\m\})$ of the punctured spectrum of $R$. 
The \textit{\'etale-local Picard group} $\mathrm{Pic}^{\rm et-loc}(R)$ of $R$ is the Picard group $\mathrm{Pic}(\Spec R^{\rm h} \setminus \{\m R^{\rm h}\})$ of the punctured spectrum of the henselization $R^{\rm h}$ of $R$. 
\end{defn}

\begin{rem}
If $(R, \m)$ is a Noetherian local $k$-algebra of depth $\ge 2$ and with residue field $k$, then $\mathrm{Pic}^{\rm et-loc}(R)$ is isomorphic to $\mathrm{Pic}(\Spec \widehat{R} \setminus \{\m \widehat{R}\})$ by \cite[Chapter II Corollaire 7.2]{Bo}. 
\end{rem}

Next we show that ``Lefschetz principle" holds for \'etale-local Picard groups.

\begin{lem}\label{Picard Lefschetz}
Let $L/K$ be an extension of algebraically closed fields of characteristic zero. 
Pick polynomials $g_1, \dots, g_s \in (x_1, \dots, x_n) K[x_1, \dots, x_n]$, and denote 
\begin{align*}
(A_{\m}, \m A_{\m})&=K[x_1, \dots, x_n]_{(x_1, \dots, x_n)}/(g_1, \dots, g_s), \\
(B_{\mathfrak{n}}, \mathfrak{n} B_{\mathfrak{n}})&=L[x_1, \dots, x_n]_{(x_1, \dots, x_n)}/(g_1, \dots, g_s).
\end{align*} 
We suppose that $\depth A_{\m} \ge 2$ and $H^2_{\m A_{\m}}(A_{\m})$ is a finite-dimensional $K$-vector space. 
\begin{enumerate}
    \item The torsion subgroup of $\mathrm{Pic}^{\rm et-loc}(A_{\m})$ is isomorphic to that of $\mathrm{Pic}^{\rm et-loc}(B_{\mathfrak{n}})$. 
    \item If $\depth A_{\m} \ge 3$, then $\mathrm{Pic}^{\rm et-loc}(A_{\m}) \cong \mathrm{Pic}^{\rm et-loc}(B_{\mathfrak{n}})$. 
\end{enumerate}
\end{lem}

\begin{proof}
(2) immediately follows from \cite[Chapter III Proposition 2.10]{Bo}, so we will only prove (1). 
Let $A=K[x_1, \dots, x_n]/(g_1, \dots, g_s)$ and $B=L[x_1, \dots, x_n]/(g_1, \dots, g_s)$. 
Let $x_A \in \Spec A$ (resp. $x_B \in \Spec B$) be the closed point corresponding to the maximal ideal $\m:=(x_1, \dots, x_n)$ (resp. $\mathfrak{n}:=(x_1, \dots, x_n)$), and denote $U_A=\Spec A \setminus \{x_A\}$ (resp. $U_B=\Spec B \setminus \{x_B\}$). 
\begin{claim}
For all integers $j$ and $n \ge 1$, there exists a natural isomorphism 
\[
H^j(U_A \otimes_A A_{\m}^{\rm h}, \mu_n) \cong H^j(U_B \otimes_B B_{\mathfrak{n}}^{\rm h}, \mu_n), 
\]
where $A_{\m}^{\rm h}$ (resp. $B_{\mathfrak{n}}^{\rm h}$) is the (strict) henselization of $A_{\m}$ (resp. $B_{\mathfrak{n}}$). 
\end{claim}
\begin{proof}[Proof of Claim]
We may reduce to the case where  $L$ is the algebraic closure of a rational function field $K(t_1, \dots, t_m)$ over $K$. 
Let $C=K[t_1, \dots, t_m, x_1, \dots, x_n]/(g_1, \dots, g_s)$ and $U_C=\mathbb{A}^m_K \times_{\Spec K} U_A=\Spec C \setminus \Spec K[t_1, \dots, t_m]$. 
We denote by $x_C \in \Spec C$ the point corresponding to the prime ideal $(x_1, \dots, x_n)$ and by $\overline{x_C}$ the geometric point over $x_C$ corresponding to $L$. 
By the smooth base change theorem for \'etale cohomology, one has 
\[
(R^j{i_C}_*f^*\mu_n)_{\overline{x_C}} \cong (f^*R^j{i_A}_* \mu_n)_{\overline{x_C}}, 
\]
where $f:\Spec C \to \Spec A$ is a natural map and $i_A:U_A \hookrightarrow \Spec A$ (resp. $i_C:U_C \hookrightarrow \Spec C$) is a natural inclusion. 
Note that the strict henselization of $\Spec C$ at $\overline{x_C}$ is isomorphic to $\Spec B_{\mathfrak{n}}^{\rm h}$ and that of $\Spec A$ at $\overline{x_C} \to \Spec C \xrightarrow{f} \Spec A$ is isomorphic to $\Spec A_{\m}^{\rm h}$. 
Thus, 
\begin{align*}
(R^j{i_C}_*f^*\mu_n)_{\overline{x_C}} & \cong H^j(U_C \otimes_C B_{\mathfrak{n}}^{\rm h}, \mu_n) \cong H^j(U_B \otimes_B B_{\mathfrak{n}}^{\rm h}, \mu_n) ,\\
(f^*R^j{i_A}_* \mu_n)_{\overline{x_C}} & \cong H^j(U_A \otimes_A A_{\m}^{\rm h}, \mu_n),
\end{align*}
so we obtain the assertion. 
\end{proof}

By virtue of \cite[Chapter II Th\'eor\`eme 7.8]{Bo}, the local Picard scheme $\mathbf{Pic}^{\rm loc}_{A_{\m}/K}$ of $A_{\m}$ over $K$ and the local Picard scheme $\mathbf{Pic}^{\rm loc}_{B_{\mathfrak{n}}/L}$ of $B_{\mathfrak{n}}$ over $L$ exist (see \cite[Chapter II]{Bo} for the definition of local Picard schemes). It then follows from the proof of  \cite[Chapter III Proposition 2.7]{Bo} that $\mathbf{Pic}^{\rm loc}_{B_{\mathfrak{n}}/L} \cong \mathbf{Pic}^{\rm loc}_{A_{\m}/K} \otimes_K L$. 
Since $L$ and $K$ are algebraically closed fields, one has a natural inclusion 
\[
\mathrm{Pic}^{\rm et-loc}(A_{\m}) \cong \mathbf{Pic}^{\rm loc}_{A_{\m}/K}(K) \hookrightarrow \mathbf{Pic}^{\rm loc}_{B_{\mathfrak{n}}/L}(L) \cong \mathrm{Pic}^{\rm et-loc}(B_{\mathfrak{n}}).
\]
Applying the above claim to the Kummer sequence, we see that this inclusion of local Picard groups 
induces an isomorphism of $n$-torsion subgroups 
\[
{}_n\mathrm{Pic}^{\rm et-loc}(A_{\m}) \cong {}_n \mathrm{Pic}^{\rm et-loc}(B_{\mathfrak{n}})
\]
for all integers $n \ge 1$, because $\mathrm{Pic}^{\rm et-loc}(A_{\m})=\mathrm{Pic}(U_A \otimes_A A_{\m}^{\rm h})$ (resp. $\mathrm{Pic}^{\rm et-loc}(B_{\mathfrak{n}})=\mathrm{Pic}(U_B \otimes_B B_{\mathfrak{n}}^{\rm h})$). 
This means that the torsion subgroup of $\mathrm{Pic}^{\rm et-loc}(A_{\m})$ is isomorphic to that of $\mathrm{Pic}^{\rm et-loc}(B_{\mathfrak{n}})$. 
\end{proof}

Now we state one of the main results of this paper. Example \ref{h2} demonstrates that the assumptions of Theorem \ref{3rd vanishing in char 0} are optimal. 
\begin{thm}\label{3rd vanishing in char 0}
Let $(S, \m)$ be an $n$-dimensional regular local ring essentially of finite type over a field of characteristic zero and let $\ba$ be an ideal of $S$. 
Suppose that  $\depth S/\ba \ge 2$ and $H^2_{\m}(S/\ba)$ is an $S/\m$-vector space. 
Then $\mathrm{cd}(S, \ba) \le n-3$ if and only if the torsion subgroup of $\mathrm{Pic}^{\rm loc}(\widehat{(\widehat{S/\ba})^{\rm sh}})$ is finitely generated. 
\end{thm}

\begin{proof}
Let $k$ be the residue field of $S$, and denote by $\overline{k}$ the algebraic closure of $k$. 
Fix a nonzero element $f \in \ba \cap \m^2$.  By \cite[Theorem 1.3]{Du}, we can find a 
regular local ring $(T_0, \m_{T_0})$ with residue field $k$ and an element $g \in T_0$ such that : (1) $T_0$ is the localization of a polynomial ring at a maximal ideal, and  (2) $T_0 \subset S$ is a faithfully flat extension which induces an isomorphism $T_0/(g) \cong S/(f)$. Let $\phi:T_0 \to T_0/(g) \to S/(f)$ be a ring homomorphism induced by the isomorphism in (2) and let $\ba_{T_0}$ be the ideal 
$\phi^{-1}(\ba S/(f)) \subset T_0$. Then  $T_0/\ba_{T_0} \cong S/\ba$.
Let $(T, \m_{T})$ be a faithfully flat extension of $T_0$ obtained by localizing the polynomial ring $\overline{k}[x_1, \dots, x_n]$ at a maximal ideal lying over $\m_{T_0}$, and set $\ba_{T}=\ba_{T_0} T$. 
Then $\mathrm{cd}(S, \ba)=\mathrm{cd}(T_0, \ba_{T_0})=\mathrm{cd}(T, \ba_{T})$. 
Since the completion of the strict henselization of the completion of $S/\ba$ is isomorphic to the completion of $T/\ba_{T}$, one has $\mathrm{Pic}^{\rm loc}(\widehat{(\widehat{S/\ba})^{\rm sh}}) \cong \mathrm{Pic}^{\rm et-loc}(T/\ba_{T})$. 
Therefore, it is enough to show that $\mathrm{cd}(T, \ba_T) \le n-3$ if and only if  the torsion subgroup of $\mathrm{Pic}^{\rm et-loc}(T/\ba_{T})$ is finitely generated. 
It is easy to see that $\depth T/\ba_{T} \ge 2$ and $H^2_{\m_{T}}(T/\ba_{T})$ is a $T/\m_{T}$-vector space, so we may assume that $S$ is the localization of the polynomial ring $\overline{k}[x_1, \dots, x_n]$ at the maximal ideal $(x_1, \dots, x_n)$. 
We remark that $\mathrm{Pic}^{\rm loc}(\widehat{(\widehat{S/\ba})^{\rm sh}}) \cong \mathrm{Pic}^{\rm et-loc}(S/\ba)$ in this case.  

Consider the subfield $k'$ of $k$ obtained by adjoining to $\Q$ all the coefficients of a set of generators $f_1, \dots, f_r$ of $\ba$. 
By a standard argument, there exists a subfield of $\C$ isomorphic to $k'$, 
so the algebraic closure $\overline{k'}$ of $k'$ can be embedded in $\C$. 
Since the $f_i$ are defined over $\overline{k'}$, set $(S_{\overline{k'}}, \m_{\overline{k'}})=\overline{k'}[x_1, \dots, x_n]_{(x_1, \dots, x_n)}$ and $\ba_{\overline{k'}}=(f_1, \dots, f_r) S_{\overline{k'}}$. 
Similarly, set $(S_{\C}, \m_{\C})=\C[x_1, \dots, x_n]_{(x_1, \dots, x_n)}$ and $\ba_{\C}=(f_1, \dots, f_r) S_{\C}$.
Then $\mathrm{cd}(S_{\C}, \ba_{\C})=\mathrm{cd}(S_{\overline{k'}}, \ba_{\overline{k'}})=\mathrm{cd}(S, \ba)$ and $\depth S_{\C}/\ba_{\C}=\depth S_{\overline{k'}}/\ba_{\overline{k'}}=\depth S/\ba \ge 2$. 
Also, it is easy to check that $H^2_{\m_{\C}}(S_{\C}/\ba_{\C})$ is a $\C$-vector space. 
Thus, we can reduce the problem to the case where $k=\C$ with the aid of Lemma \ref{Picard Lefschetz}.  

From now on, we consider the case where $S=\C[x_1, \dots, x_n]_{(x_1, \dots, x_n)}$. 
Since $\depth S/\ba \ge 2$, we know from \cite[Corollary 2.11]{Og} that $\mathrm{cd}(S, \ba) \le n-2$. 
Therefore, $\mathrm{cd}(S, \ba)\le n-3$ if and only if $H^{n-2}_{\ba}(S)=0$. 

\begin{claim}
$H^{n-2}_{\ba}(S)$ is supported only at the maximal ideal $\m$. 
\end{claim}
\begin{proof}[Proof of Claim]
We may assume that $S$ is a complete regular local ring for the proof of Claim. 
We will show that $H^{n-2}_{\ba S_{\fp}}(S_{\fp})=0$ for all prime ideals $\fp$ of height $n-1$. 
By the second vanishing theorem \cite[Corollary 2.11]{Og}, it is enough to show that $\depth S_{\fp}/\ba S_{\fp} \ge 2$ for all prime ideals $\fp$ of height $n-1$ containing $\ba$.  
Since $H^2_{\m}(S/\ba)$ is an $S/\m$-vector space, $\mathrm{Ext}^{n-2}_S(S/\ba, S)$ is also an $S/\m$-vector space by local duality.  
Similarly, since $\depth S/\ba \ge 2$, we see by local duality that $\mathrm{Ext}^{n-1}_S(S/\ba, S)=0$.
Thus, $\mathrm{Ext}^{n-2}_{S_{\fp}}(S_{\fp}/\ba S_{\fp}, S_{\fp})=\mathrm{Ext}^{n-1}_{S_{\fp}}(S_{\fp}/\ba S_{\fp}, S_{\fp})=0$ for all prime ideals $\fp$ of height $n-1$.
Applying local duality again, one has that $H^1_{\fp S_{\fp}}(S_{\fp}/\ba S_{\fp})=H^0_{\fp S_{\fp}}(S_{\fp}/\ba S_{\fp})=0$. 
\end{proof}

Let $f_1, \dots, f_r \in (x_1, \dots, x_n)\C[x_1, \dots, x_n]$ be a set of generators of $\ba$ and $X$ be the closed subscheme of $\C^n$ defined by the $f_i$. 
If $x \in X$ denotes the origin of $\C^n$, then $\sO_{X, x} \cong S/\ba$. 
It follows from \cite[Theorem 2.8]{Og} and the above claim that the vanishing of $H^{n-2}_{\ba}(S)$ is equivalent to the vanishing of the local de Rham cohomology $H^2_{\{x\}, {\rm dR}}(\Spec \widehat{\sO}_{X,x})$, where $\widehat{\sO}_{X, x}$ is the completion of $\sO_{X, x}$. 
This local de Rham cohomology is isomorphic to the relative singular cohomology $H^2(X_{\rm an}, (X \setminus\{x\})_{\rm an}; \C)$ by \cite[Chapter III Proposition 3.1, Chapter IV Theorem 1.1]{Ha}.
Since the homology groups of a complex quasi-projective scheme (with coefficients in $\Z$) are all finitely generated by \cite[Chapter 1 (6.10)]{Di}, the relative homology $H_i(X_{\rm an}, (X \setminus\{x\})_{\rm an}; \Z)$ is also finitely generated for all $i$. 
It then follows from the universal coefficient theorem that 
\[H^2(X_{\rm an}, (X \setminus\{x\})_{\rm an}; \C) \cong H^2(X_{\rm an}, (X \setminus\{x\})_{\rm an}; \Z) \otimes \C.\] 
Thus, $H^{n-2}_{\ba}(S)=0$ if and only if $H^2(X_{\rm an}, (X \setminus\{x\})_{\rm an}; \Z)$ is torsion. 
We will show that $H^2(X_{\rm an}, (X \setminus\{x\})_{\rm an}; \Z)$ is torsion if and only if the torsion subgroup of $\mathrm{Pic}^{\rm et-loc}(S/\ba)$ is finitely generated. 

Let $U \subset X$ be a contractible Stein open neighborhood of $x$. It follows from the excision theorem \cite[Chapter II Lemma 1.1]{BS} that for each $i \in \N$, 
\begin{align*}
H^{i+1}(X_{\rm an}, (X \setminus\{x\})_{\rm an}; \Z) & \cong H^{i+1}(U, U \setminus \{x\}; \Z) \cong H^i(U \setminus \{x\}, \Z),\\
H^{i+1}_{\{x\}}(X_{\rm an}, \sO_{X_{\rm an}}) & \cong H^{i+1}_{\{x\}}(U, \sO_{U}) \cong H^i(U \setminus \{x\}, \sO_U).
\end{align*}
First we consider the following commutative diagram: 
\[
\xymatrix{
H^0(U, \sO_U^{\times}) \ar[r] \ar[d]_{\rho^{\times}_{U, U \setminus \{x\}}} & H^0(U, \sO_U) \ar[d]^{\rho^{}_{U, U \setminus \{x\}}}\\
H^0(U \setminus\{x\}, \sO_U^{\times}) \ar[r] & H^0(U \setminus\{x\}, \sO_U),
}
\]
where the vertical maps are the restriction maps of sections and the horizontal maps are injections induced by the inclusion of sheaves $\sO_U^{\times} \hookrightarrow \sO_U$. 
Note that the restriction map $\rho_{U, U \setminus \{x\}}$ is surjective by \cite[Chapter II Theorem 3.6]{BS}, because $\depth \sO_{U, x} \ge 2$. 
Let $s \in H^0(U \setminus \{x\}, \sO_U^{\times}) \subseteq H^0(U \setminus \{x\}, \sO_U)$. 
Since $\rho_{U, U \setminus \{x\}}$ is surjective, there exists an extension $\widehat{s} \in H^0(U, \sO_U)$ of $s$ that does not vanish on $U \setminus \{x\}$.   
If $\widehat{s}$ is not in $H^0(U, \sO_U^{\times})$, then $\widehat{s}(x)= 0$, which implies that $\dim_x U \le 1$. 
This contradicts the assumption that $\depth \sO_{U, x} \ge 2$. 
Thus, $\widehat{s} \in H^0(X, \sO_X^{\times})$, that is, $\rho^{\times}_{U, U \setminus \{x\}}$ is surjective. 

Next we consider the following commutative diagram with exact rows, induced by the exponential sequence\footnote{The exponential sequence $0 \to \Z \xrightarrow{\times 2\pi i} \sO_X \xrightarrow{\mathbf{e}} \sO_X^{\times} \to 1$ exists on any (even non-reduced) complex analytic space $X$, where $\mathbf{e}$ is defined as follows: if $f \in \sO_X$ is locally represented as the restriction of a holomorphic function $\widetilde{f}$ on a complex manifold $M$, then $\mathbf{e}(f)=e^{2 \pi i \widetilde{f}}|_X$. It is easy to check its well-definedness and the exactness of the exponential sequence.}: 
\[
\xymatrix{
H^0(U, \sO_U) \ar[r] \ar[d]^{\rho_{U, U \setminus \{x\}}} & H^0(U, \sO_U^{\times}) \ar[r] \ar[d]^{\rho^{\times}_{U, U \setminus \{x\}}} & H^1(U, \Z) \ar[d] &  \\
H^0(U \setminus \{x\}, \sO_U) \ar[r] & H^0(U \setminus \{x\}, \sO_U^{\times}) \ar[r] &  H^1(U \setminus \{x\}, \Z) \ar[r] & H^1(U \setminus \{x\}, \sO_U).\\
}
\]
Since $U$ is contractible, the map $H^0(U, \sO_U) \to H^0(U, \sO_U^{\times})$ is surjective. 
Combining with the surjectivity of $\rho^{\times}_{U, U \setminus \{x\}}$, we see that $H^0(U \setminus \{x\}, \sO_U) \to H^0(U \setminus \{x\}, \sO_U^{\times})$ is also surjective, which is equivalent to saying that $H^1(U \setminus \{x\}, \Z) \to H^1(U \setminus \{x\}, \sO_U)$ is injective. 
Therefore, we obtain the following exact sequence from the exponential sequence:
\[
0 \to H^1(U \setminus \{x\}, \Z) \to H^1(U \setminus \{x\}, \sO_U) \to H^1(U \setminus \{x\}, \sO_U^{\times}) \to H^2(U \setminus \{x\}, \Z). 
\]
We then use the fact that the \'etale-local Picard group $\mathrm{Pic}^{\rm et-loc}(S/\ba)$ is isomorphic to 
the direct limit of $\mathrm{Pic}(U \setminus \{x\}, \sO_U^{\times})$ as $U \subset X$ runs through all analytic open neighborhoods of $x$, 
which follows from \cite[Chapter III Proposition 3.2]{Bo} and the proof of Proposition 3.6 in \textit{loc.~cit}. 
Taking the direct limit of the above exact sequence, we have the following exact sequence: 
\[
0 \to H^{2}(X_{\rm an}, (X \setminus\{x\})_{\rm an}; \Z) \to H^2_{\{x\}}(X_{\rm an}, \sO_{X_{\rm an}}) \to \mathrm{Pic}^{\rm et-loc}(S/\ba) \to H^{3}(X_{\rm an}, (X \setminus\{x\})_{\rm an}; \Z). 
\]
Note that $H^2_{\{x\}}(X_{\rm an}, \sO_{X_{\rm an}})$ is a $\C$-vector space. 
If $H^{2}(X_{\rm an}, (X \setminus\{x\})_{\rm an}; \Z)$ is not torsion, then $\mathrm{Pic}^{\rm et-loc}(S/\ba)$ contains $\Q/\Z$, an infinitely generated torsion group.   
Conversely, if $H^{2}(X_{\rm an}, (X \setminus\{x\})_{\rm an}; \Z)$ is torsion, then it has to be zero, 
and the torsion subgroup of $\mathrm{Pic}^{\rm et-loc}(S/\ba)$ is isomorphic to a subgroup of $H^{3}(X_{\rm an}, (X \setminus\{x\})_{\rm an}; \Z)$, which is finitely generated. Thus, we complete the proof of Theorem \ref{3rd vanishing in char 0}. 
\end{proof}

\begin{rem}\label{remark in positive char}
An analogous statement to Theorem \ref{3rd vanishing in char 0} does not hold in positive characteristic. For example, let $E$ be a supersingular elliptic curve in the projective plane $\pP^2$ over an algebraically closed field $k$ of characteristic $p>0$ and let $(R, \m)$ be the localization of the affine cone over $E \times E$ at the unique homogeneous maximal ideal. 
We easily see that $\depth R=2$ and the natural Frobenius action on $H^2_{\m}(R)$ is nilpotent, because $E$ is a supersingular elliptic curve. 
$R$ has embedding dimension $9$, so let $S=k[x_1, \dots, x_9]_{(x_1, \dots, x_9)}$ and $\ba$ be the kernel of the natural surjection $S \to R$. 
Since $H^0_{\m}(R)=H^1_\m(R)=0$ and the Frobenius action on $H^2_{\m}(R)$ is nilpotent, it follows from \cite[Corollary 3.2]{Lyu3} that $H^{9}_{\ba}(S)=H^{8}_{\ba}(S)=H^{7}_{\ba}(S)=0$, that is, $\mathrm{cd}(S, \ba)=6$, as the height of $\ba$ is equal to $6$. 

On the other hand, since $R$ is a normal isolated singularity, one has an inclusion 
\[\mathrm{Pic}(E \times E)/\Z \cong \mathrm{Cl}(R) \hookrightarrow \mathrm{Cl}(\widehat{R}) \cong \mathrm{Pic}^{\rm loc}(\widehat{R}), 
\] 
where the last isomorphism follows from \cite[Proposition 18.10]{Fo}. 
Thus, the torsion subgroup of $\mathrm{Pic}^{\rm loc}(\widehat{(\widehat{S/\ba})^{\rm sh}})=\mathrm{Pic}^{\rm loc}(\widehat{R})$ is not finitely generated, because the torsion subgroup of the Picard group of the abelian variety $E \times E$ is not finitely generated.  
\end{rem}

\begin{eg}\label{h2}
In Theorem \ref{3rd vanishing in char 0}, the assumption on $H^2_{\m}(R)$ cannot be removed. Let $S=\C[x,y,u,v,w]_{(x, y, u, v, w)}$ and $(R, \m)=S/\ba$ where $\ba=(x,y)\cap(u,v)$. Then $\depth R=2$ but $H^2_{\m}(R)$ does not have finite length. For suppose it does, then local duality would imply that $\depth R_{\fp}\geq 2$, where $\fp = (x,y,u,v)$. Since $R_{\fp}$ has disconnected punctured spectrum, this is a contradiction. 
On the other hand, 
$\mathrm{Pic}^{\rm loc}(\widehat{(\widehat{S/\ba})^{\rm sh}}) = \mathrm{Pic}^{\rm loc}(\widehat{R}) = \Z$ (the proof is the same as that of \cite[Example 28, 29]{Ko}, or see \cite[Example 5.3]{Ha3}). 
However, since $\ba$ is a square-free monomial ideal, $\cd(S,\ba) = \pd_S R = 5-2$. 
Thus, the conclusion of Theorem \ref{3rd vanishing in char 0} does not hold if we remove the condition that $H^2_{\m}(R)$ has finite length. 
\end{eg}

\begin{prop}[\textup{cf.~\cite[Lemma 8]{Ko2}}]
\label{depth3}
Let $(S, \m)$ be an $n$-dimensional regular local ring essentially of finite type over a field of characteristic zero and let $\ba$ be an ideal of $S$. 
If $\depth S/\ba \ge 3$, then $\mathrm{Pic}^{\rm loc}(\widehat{(\widehat{S/\ba})^{\rm sh}})$ is finitely generated. 
\end{prop}

\begin{proof}
We use the same strategy as the proof of Theorem \ref{3rd vanishing in char 0}. 
We may assume that $S=\C[x_1, \dots, x_n]_{(x_1, \dots, x_n)}$. 
Let $x \in X$ be a closed point of an affine scheme $X$ of finite type over $\C$ such that $\sO_{X, x} \cong S/\ba$. 
The exponential sequence then induces the following exact sequence:
\[
H^{2}(X_{\rm an}, (X \setminus\{x\})_{\rm an}; \Z) \to H^2_{\{x\}}(X_{\rm an}, \sO_{X_{\rm an}}) \to \mathrm{Pic}^{\rm loc}(\widehat{S/\ba}) \to H^{3}(X_{\rm an}, (X \setminus\{x\})_{\rm an}; \Z).  
\]
If $\depth S/\ba \ge 3$, then we see that $H^2_{\{x\}}(X_{\rm an}, \sO_{X_{\rm an}})$ vanishes and then $\mathrm{Pic}^{\rm loc}(\widehat{S/\ba})$ is isomorphic to a subgroup of $H^{3}(X_{\rm an}, (X \setminus\{x\})_{\rm an}; \Z)$.
Thus, $\mathrm{Pic}^{\rm loc}(\widehat{S/\ba})$ is finitely generated. 
\end{proof}

\begin{cor}\label{vanishing}
Let $(S,\m)$ be an $n$-dimensional regular local ring essentially of finite type over a field. 
If $\ba$ is an ideal of $S$ such that $\depth S/\ba \ge 3$, then $\mathrm{cd}(S, \ba) \le n-3$. 
\end{cor}

\begin{proof}
The positive characteristic case follows from a result of Peskine and Szpiro \cite{PS} and the characteristic zero case does from Theorem \ref{3rd vanishing in char 0} and Proposition \ref{depth3}. 
\end{proof}

When $\depth S/\ba \ge 4$, the cohomological dimension $\mathrm{cd}(S, \ba)$ is not necessarily less than $n-3$ (see Example \ref{segre}). 
We give a necessary and sufficient condition for $\mathrm{cd}(S, \ba)$ to be less than $n-3$ in terms of the local Picard group of the completion of the strict henselization of the completion of $S/\ba$.  

\begin{thm}\label{picard group}
Let $(S,\m)$ be an $n$-dimensional regular local ring essentially of finite type over a field of characteristic zero and let $\ba$ be an ideal of $S$ such that $\depth S/\ba \ge 4$. 
Then $\mathrm{cd}(S, \ba) \le n-4$ if and only if $\mathrm{Pic}^{\rm loc}(\widehat{(\widehat{S/\ba})^{\rm sh}})$ is torsion. 
\end{thm}

\begin{proof}
We use the same strategy as the proof of Theorem \ref{3rd vanishing in char 0} again. 
We may assume that $S=\C[x_1, \dots, x_n]_{(x_1, \dots, x_n)}$. 
Note that $H^{i}_{\ba}(S)=0$ for all $i \ge n-2$ by Corollary \ref{vanishing}. 
We also see from an argument analogous to Claim in the proof of Theorem \ref{3rd vanishing in char 0} that $H^{n-3}_{\ba}(S)$ is supported only at the maximal ideal $\m$. 
Let $x \in X$ be a closed point of an affine scheme $X$ of finite type over $\C$ such that $\sO_{X, x} \cong S/\ba$. 
The vanishing of $H^{n-3}_{\ba}(S)$ is equivalent to saying that the relative singular cohomology $H^3(X_{\rm an}, (X \setminus\{x\})_{\rm an}; \Z)$ is torsion.  
We will show that $H^3(X_{\rm an}, (X \setminus\{x\})_{\rm an}; \Z) \cong \mathrm{Pic}^{\rm et-loc}(S/\ba)$. 

Let $U \subset X$ be a contractible Stein open neighborhood of $x$. 
It follows from the excision theorem and the contractibility of $U$ that 
\[H^3(X_{\rm an}, (X \setminus\{x\})_{\rm an}; \Z) \cong H^3(U, U \setminus \{x\}; \Z) \cong H^2(U \setminus \{x\}, \Z).\]
Also, since $\depth \sO_{U, x} \ge 4$ and $U$ is Stein, one has $H^i(U \setminus \{x\}, \sO_U) \cong H^{i+1}_{\{x\}}(U, \sO_U)=0$ for $i=1, 2$. 
Then the exponential sequence induces the following exact sequence: 
\[
0=H^1(U \setminus \{x\}, \sO_U) \to H^1(U \setminus \{x\}, \sO_U^{\times}) \to H^2(U \setminus \{x\}, \Z) \to H^2(U \setminus \{x\}, \sO_U)=0.
\]
In other words, $H^1(U \setminus \{x\}, \sO_U^{\times}) \cong H^2(U \setminus \{x\}, \Z)$. 
Thus, we can conclude that 
\[\mathrm{Pic}^{\rm et-loc}(S/\ba) \cong  \varinjlim_{U \ni x} H^1(U \setminus \{x\}, \sO_U^{\times}) \cong \varinjlim_{U \ni x} H^2(U \setminus \{x\}, \Z) \cong H^3(X_{\rm an}, (X \setminus\{x\})_{\rm an}; \Z).\] 
\end{proof}

The following corollary is immediate from Corollary \ref{vanishing} and Theorem \ref{picard group}. 
\begin{cor}\label{n-3}
Let $(S,\m)$ be an $n$-dimensional regular local ring essentially of finite type over a field of characteristic zero and let $\ba$ be an ideal of $S$ such that $\depth S/\ba \ge 4$. Then $\cd(S,\ba)=n-3$ if and only if $\mathrm{Pic}^{\rm loc}(\widehat{(\widehat{S/\ba})^{\rm sh}})$ is not torsion.
\end{cor}

\begin{eg}\label{segre}
Let $A$ and $B$ be standard graded Cohen-Macaulay normal domains over an algebraically closed field of characteristic zero. 
Suppose that $\dim A\geq 3$, $\dim B\geq 2$, and the $a$-invariants of $A$ and $B$ are both negative. 
Let $R=A\# B$ be the Segre product of $A$ and $B$. 
We write $R=T/\ba$ using the standard embedding, 
where $T$ is a standard graded polynomial ring with unique homogeneous maximal ideal $\m$. 
Let $S=T_{\m}$.  
Then $\cd(S,\ba) = \dim S-3$. 

To prove this, we just need to verify the conditions of Corollary \ref{n-3}. We know that $\dim R=\dim A+\dim B-1\geq 4$ and $R$ is a Cohen-Macaulay normal domain (see \cite[Theorem 4.2.3]{GW}). Let $U=\Proj A$, $V=\Proj B$ and $X=\Proj R$. Then $\Pic(X)= \Pic(U)\times \Pic(V)$ has rank at least $2$. 
For this we need the assumption that  $A$ has depth at least $3$ (\cite[Exercise III.12.6]{Ha1}). 
Since $R$ is normal, there exists the following exact sequence (see \cite[Exercise II.6.3]{Ha1})
\[0 \to \Z \to \Cl(X) \to \Cl(\Spec R \setminus \{\m R\}) \to 0,\]
which induces an exact sequence of Picard groups
\[0 \to \Z \to \Pic(X) \to \Pic(\Spec R \setminus \{\m R\}) \to 0.\]
It then follows that $\mathrm{Pic}^{\rm loc}(\widehat{(\widehat{S/\ba})^{\rm sh}})=\mathrm{Pic}^{\rm loc}(\widehat{S/\ba})$ has positive rank. Thus $\cd(S,\ba) = \dim S-3$. 

For instance, when $A$ and $B$ are polynomial rings of dimension $m,n$, respectively, we see that $\ba$ is generated by $2 \times 2$ minors of the $m\times n$ matrix of indeterminates, and it is well-known that in such a case $\cd(S,\ba)=mn-3$ (see for example \cite[Remark 7.12]{BrVe}). 
\end{eg}

%%%%%%%%%%%%%%%%%%%%%%%%%%%%%%%%%%%%%%%%%%%%%%%%%%%%%%%%%%%%%%%%%%%%%%%%%%%%%%%%%%%%%%%%%%%%%%%%%%%%%%%%%%%%%%%%%%%%%%%%%%%%%%%%%%%%%%%%%%%%%%%%%%%%%%%%%%%%%%%%%%%%%%%%%%%%%%%%%%%%%%%%%%%%%%%%%%%%%%%%%%%%%%%%%%%%%%%%%%

\section{Applications}

In this section, we give a number of applications of the results in $\S \ref{vanishing section}$. First, we recall the definition of Lyubeznik numbers. 

\begin{defn}[\textup{\cite[Definition 4.1]{Lyu}}]\label{Lyubeznik def}
 Let $R$ be a Noetherian local ring that admits a surjection from an $n$-dimensional regular local ring $(S, \m)$ of equal characteristic. 
 Let $\ba \subset S$ be the kernel of the surjection and let $k = S/\m$ be the residue field of $S$. 
 Then for each $0 \le i, j \le n$, the \textit{Lyubeznik number} $\lambda_{i,j}(R)$ is defined by 
 \[
 \lambda_{i,j}(R)=\dim_k(\mathrm{Ext}^i_S(k,H^{n-j}_{\ba}(S))).
 \] 
It is known that the $\lambda_{i,j}(R)$ are all finite and independent of the choice of the surjection $S \to R$. 
\end{defn}

As an application of Corollary \ref{vanishing}, we obtain vanishing results of Lyubeznik numbers. 

\begin{prop}\label{lyubeznik number}
Let $R$ be a local ring essentially of finite type over a field. 
Then for all $j <\depth R$, one has 
\[\lambda_{j-2,j}(R)=\lambda_{j-1,j}(R)=\lambda_{j,j}(R)=0.\]
\end{prop}

\begin{proof}
Let $S$ be an $n$-dimensional regular local ring essentially of finite type over a field and $\ba$ be an ideal of $S$ such that $S/\ba \cong R$. 
Since the injective dimension of $H^{n-j}_{\ba}(S)$ is less than or equal to the dimension of the support of $H^{n-j}_{\ba}(S)$ by \cite[Corollary 3.6 (b)]{Lyu}, we will show that 
\[
\dim \mathrm{Supp}(H^{n-j}_{\ba}(S)) \le j-3
\]
for all $0 \le j<\depth S/\ba$ (here we use the convention that the dimension of the empty set is $-\infty$). 
To check this, it is enough to show that $H^{n-j}_{\ba S_{\fp}}(S_{\fp})=0$ for all $0 \le j <\depth S/\ba$ and for all prime ideals $\fp$ of height $h \le n-j+2$. 
If $j=0$, then $H^n_{\ba}(S)=0$ by the Lichtenbaum-Hartshorne vanishing theorem \cite{Ha2}. 
If $j=1$, then $H^{n-1}_{\ba}(S)=0$ by the second vanishing theorem \cite{Og}, \cite{PS}.  
Thus, we may assume that $j \ge 2$ and $h=n-j+2$. 

Since $\depth S/\ba > j \ge 2$, we have $H^j_{\m}(S/\ba)=H^{j-1}_{\m}(S/\ba)=H^{j-2}_{\m}(S/\ba)=0$. 
Local duality over $S$ yields that $\mathrm{Ext}^{n-j}_{S}(S/\ba, S)=\mathrm{Ext}^{n-j+1}_{S}(S/\ba, S)=\mathrm{Ext}^{n-j+2}_{S}(S/\ba, S)=0$.  
In particular, 
$\mathrm{Ext}^{n-j}_{S_{\fp}}(S_{\fp}/\ba S_{\fp}, S_{\fp})=\mathrm{Ext}^{n-j+1}_{S_{\fp}}(S_{\fp}/\ba S_{\fp}, S_{\fp})=\mathrm{Ext}^{n-j+2}_{S_{\fp}}(S_{\fp}/\ba S_{\fp}, S_{\fp})=0$.
Then local duality over the $(n-j+2)$-dimensional regular local ring $S_{\fp}$ yields that 
\[
H^2_{\fp S_{\fp}}(S_{\fp}/\ba S_{\fp})=H^1_{\fp S_{\fp}}(S_{\fp}/\ba S_{\fp})=H^0_{\fp S_{\fp}}(S_{\fp}/\ba S_{\fp})=0,
\] 
that is, $\depth S_{\fp}/\ba S_{\fp} \ge 3$. 
We, therefore, conclude from Corollary \ref{vanishing} that $H^{n-j}_{\ba S_{\fp}}(S_{\fp})=0$.
\end{proof}

The following proposition comes from a discussion with Matteo Varbaro, whom we thank.
\begin{prop}\label{S3}
Let $R$ be a $d$-dimensional local ring essentially of finite type over a field.
If $R$ satisfies Serre's condition $(S_3)$, then $\lambda_{d-1, d}(R)=0$. 
\end{prop}

\begin{proof}
Let $(S,\m)$ be an $n$-dimensional regular local ring essentially of finite type over a field and $\ba$ be an ideal of $S$ such that $S/\ba \cong R$. 
We use the Grothendieck spectral sequence 
\[
E^{p, q}_2=H^p_{\m}(H^{q}_{\ba}(S)) \Longrightarrow E^{p+q}=H^{p+q}_\m (S). 
\] 
Since $E^{d-1, n-d}_2$ is an injective $S$-module by \cite[Corollary 3.6 (a)]{Lyu}, it is isomorphic to the direct sum of $\lambda_{d-1,d}(R)$ copies of the injective hull $E_S(S/\m)$ of the residue field of $S$. 
In particular, the vanishing of $\lambda_{d-1,d}(R)$ is equivalent to saying that $E^{d-1, n-d}_2=0$. 

\begin{claim}
$E^{d-r-1, n-d+r-1}_2=0$ for all $r \ge 2$. 
\end{claim}
\begin{proof}[Proof of Claim]
We may assume that $d \ge 3$ and $r \le d+1$. 
Since the injective dimension of $H^{n-d+r-1}_{\ba}(S)$ is less than or equal to the dimension of the support of $H^{n-d+r-1}_{\ba}(S)$ by \cite[Corollary 3.6 (b)]{Lyu}, it suffices to show that 
\[\dim \mathrm{Supp}(H^{n-d+r-1}_{\ba}(S)) \le d-r-2.\] 
In other words, we will show that $H^{n-d+r-1}_{\ba S_{\fp}}(S_{\fp})=0$ for all prime ideals $\fp$ of height $h \le n-d+r+1$. 
Note that $H^n_{\ba}(S)=H^{n-1}_{\ba}(S)=0$ by the Lichtenbaum-Hartshorne vanishing theorem \cite{Ha2} and the second vanishing theorem \cite{Og}, \cite{PS}.  
Therefore, we only consider the case where $r \le d-1$ and $h=n-d+r+1$. 
In this case, $\depth S_{\fp}/ \ba S_{\fp} \ge 3$ by assumption. 
Applying Corollary \ref{vanishing} to the $(n-d+r+1)$-dimensional local ring $S_{\fp}$, we see that $H^{n-d+r-1}_{\ba S_{\fp}}(S_{\fp})=0$. 
\end{proof}
On the other hand, it is easy to see that $E^{d+r-1, n-d-r+1}_2=0$ for all $r \ge2$. Combining this with the above claim, we conclude that 
\[E^{d-1, n-d}_2 \cong E^{d-1, n-d}_3 \cong \cdots \cong E^{d-1, n-d}_{\infty}=0,\]
where the last equality follows from the fact that $E^{n-1}=H^{n-1}_\m(S)=0$. 
\end{proof}

\begin{rem}
Let $R$ be an excellent local ring containing a field (but not necessarily essentially of finite type over a field). If $R$ is an isolated singularity,  
then making use of Artin's algebraization theorem \cite[Theorem 3.8]{Ar}, we still have the same vanishing results as Propositions \ref{lyubeznik number} and \ref{S3}.  
\end{rem}

Next, we prove an extension of a recent result by Katzman-Lyubeznik-Zhang \cite{KLZ} which in turn extended the classical result of Hartshorne on connectedness of the punctured spectrum. 

\begin{defn}[\textup{\cite[Definition 1.1]{KLZ}}]
Let $(R, \m)$ be a Noetherian local ring and let $\fp_1, \dots, \fp_m$ be the minimal primes of $R$. 
The simplicial complex $\Delta(R)$ associated to $R$ is a simplicial complex on the vertices $1, \dots, m$ such that a simplex $\{i_0, \dots, i_s\}$ is included in $\Delta(R)$ if and only if $\fp_{i_0}+\cdots+\fp_{i_s}$ is not $\m$-primary. 
\end{defn}

\begin{prop}\label{KLZex}
Let $(R,\m, k)$ be a local ring essentially of finite type over a field with residue field $k$ separably closed. 
If $\depth R \ge 3$, then 
\[
\widetilde{H}_0(\Delta(\widehat{R}); k)=\widetilde{H}_1(\Delta(\widehat{R}); k)=0, 
\]
where the $\widetilde{H}_i(\Delta(\widehat{R}); k)$ are the reduced singular homology of $\Delta(\widehat{R})$. 
\end{prop}

\begin{proof}
It follows from a very similar argument to the proof of \cite[Theorem 1.2]{KLZ}, together with Corollary \ref{vanishing}. 
\end{proof}

Finally, we note some consequences of our results combined with the key induction theorem \cite[Theorem 2.5]{HL}. 
We start with a reinterpretation of this theorem which is more convenient for our use:

\begin{thm}[Huneke-Lyubeznik] \label{keyHL}
Let $(S,\m)$ be a $d$-dimensional regular local ring containing a field and let $\ba\subset S$ be an ideal of pure height $c$. 
Let $f : \N \to \N$ be a non-decreasing function. Assume there exist integers $l'\geq l\geq c$  such that

\begin{enumerate}
\item $f(l)\geq c$, 
\item $\cd(S_{\fp}, \ba_{\fp})\leq f(l'+1)-c+1$ for all prime ideals $\fp \supset \ba$ with $\height \fp \leq l-1$, 
\item $\cd(S_{\fp}, \ba_{\fp})\leq f(\height \fp)$ for all prime ideals $\fp \supset \ba$ with $l\leq \height \fp \leq l'$, 
\item $f(r-s-1)\leq f(r)-s$ for every $r\geq l'+1$ and every $c-1 \geq s \ge 1$. 
\end{enumerate}
Then $\cd(S,\ba)\leq f(d)$ if $d\geq l$. 

\begin{proof}
We apply   \cite[Theorem 2.5]{HL} with $M=A=B=S$, $I= \ba$ and $n= f(d)+1$. 
We check all the conditions of \cite[Theorem 2.5]{HL}. First, we need to show $n>c$. 
However, since $f$ is non-decreasing and $d\geq l$,  it follows from (1) that $n-1 = f(d)\geq f(l)\geq c$.  The condition (3) allows us to  assume $d\geq l'+1$ (otherwise we take $\fp=\m$ to conclude).

Next, let $s$ be an integer such that $1\leq s\leq c-1$. 
In order to verify the two conditions (i) and (ii) in \cite[Theorem 2.5]{HL}, by \cite[Lemma 2.4]{HL}, it is enough to show the following claim: 

\begin{claim}
$\cd(S_{\fp}, \ba_{\fp})\leq n-s-1$ for all prime ideals $\fp \supset \ba$ with $\dim S/\fp \ge s+1$. 
\end{claim}

Since $d\geq l'+1$, we have that $n-s-1=f(d)-s \geq f(l'+1)-s\geq f(l'+1)-c+1$, so (2) proves the claim if $\height \fp \leq l-1$. If $\height \fp\geq l$, then by (3) and induction on $d$, we know that $\cd(S_{\fp}, \ba_{\fp})\leq f(\height \fp)$. 
However, since $\height \fp \leq d-s-1$, it follows from (4) that  
\[f(\height \fp)\leq f(d-s-1)\leq f(d)-s=n-s-1.\]
\end{proof}
\end{thm}

\begin{thm}\label{serre}
Let $S$ be an $n$-dimensional regular local ring containing a field and let $\ba \subset S$ be an ideal of height $c$. 
\begin{enumerate}
\item If $S/\ba$ satisfies Serre's condition $(S_2)$ and $\dim S/\ba\geq 1$, then 
\[\cd(S,\ba)\leq n-1-\flo{n-2}{c}.\]
\item Suppose that $S$ is essentially of finite type over a field. 
If $S/\ba$ satisfies Serre's condition $(S_3)$ and $\dim S/\ba \ge 2$, then 
\[\cd(S,\ba)\leq n-2-\flo{n-3}{c}.\]
\end{enumerate} 
\end{thm}

\begin{proof}
First note that $\ba$ is of pure height $c$ because $S/\ba$ satisfies $(S_2)$.   
For the statement $(1)$, we use Theorem \ref{keyHL} with $f(m)=m-1-\flo{m-2}{c}$, $l=c+1$ and $l'=2c+1$. 
For $(2)$, we use $f(m)=m-2-\flo{m-3}{c}$, $l=c+2$ and $l'=2c+2$. 
To verify the condition $(3)$ of Theorem \ref{keyHL}, we need to invoke Corollary \ref{vanishing} as follows. 
If $\height \fp=l=c+2$, then it follows from (1) that $\cd(S_{\fp}, \ba_{\fp}) \leq c=f(\height \fp)$. 
If $c+3 \leq \height \fp \leq l'=2c+2$, then $f(\height \fp)=\height \fp-3$.
However, since $\depth S_{\fp}/\ba S_{\fp} \geq 3$ in this case,  it follows from Corollary \ref{vanishing} that $\cd(S_{\fp}, \ba_{\fp}) \leq \height \fp -3$. 
\end{proof}

\begin{rem}
There have been many results in the literature which give similar bounds when the strict henselization of $R$ is a domain or has a small number of minimal primes. For example,  \cite[Corollary 1.2]{Lyu4} gives the same bound as that of Theorem \ref{serre} (1) under the assumption that the number of minimal primes is less than $n/c$. Of course, there are a lot of examples of ideals with good depth and many minimal primes. A very elementary example is  the complete intersection $I=(x_1\dots x_n, y_1\dots y_n) \subset k[[x_1,\cdots,x_n,y_1,\cdots, y_n]]$, which has $n^2$ minimal primes $\{(x_i,y_j)\}$. 
\end{rem}

There are many examples where the bounds in Theorem \ref{serre} are sharp. For example, one can use the Segre product of two polynomial rings of dimension $2$ and $d$ with $d\geq 3$, as explained in Example \ref{segre}. However, those examples have $c$ relatively big compared to $n$. We suspect that in general one can do a little better, for instance, as follows:  

\begin{ques}
Let $S$ be an excellent regular local ring that contains a field. Let  $n=\dim S$, and $\ba \subset S$ an ideal of height $c$. If $S/\ba$ satisfies Serre's condition $(S_2)$, then is it always true that 
\[\cd(S,\ba)\leq n- \flo{n}{c+1}-\flo{n-1}{c+1}?\]
\end{ques}

Under some extra assumption, for example if $S/\ba$ is normal, the answer is yes by \cite[Theorem 3.9]{HL}. If the answer to the above question is affirmative, then one can show that the bound is sharp for {\it most} $n$ and $c$. 
We give a class of examples in the case when $c$ is odd using square-free monomial ideals. More details will be explained in \cite{DE}.

\begin{eg}
Let $c=2l-1$, $l>1$  and suppose $n=ql$. Let $S= k[x_1,\cdots, x_n]$, $\m=(x_1, \dots, x_n)$ and $J_1$ be the monomial ideal $(x_1\cdots x_l, x_{l+1}\cdots x_{2l},\dots, x_{(q-1)l+1} \cdots x_{ql})$. Let $J_2$ be the square-free part of $J_1\m^{l-1}$. Then $J_2$ is an ideal with linear presentation and all the generators are in degree $2l-1=c$. The regularity $\reg J_2$ of $J_2$ is equal to 
\[n-q+1 = n-\flo{n}{2l} - \flo{n-1}{2l} = n-\flo{n}{c+1} - \flo{n-1}{c+1}.\]
On the other hand, if $\ba$ is the Alexander dual of $J_2$, then $S/\ba$ satisfies $(S_2)$, $\height \ba =c$ and $$\cd (S,\ba) = \pd S/\ba= \reg J_2$$ (see for example \cite{DHS}).

\end{eg}

%%%%%%%%%%%%%%%%%%%%%%%%%%%%%%%%%%%%%%%%%%%%%%%%%%%%%%%%%%%%%%%%%%%%%%%%%%%%%%%%%%%%%%%%%%%%%%%%%%%%%%%%%%%%%%%%%%%%%%%%%%%%%%%%%%%%%%%%%%%%%%%%%%%%%%%%%%%%%%%%%%%%%%%%%%%%%%%%%%%%%%%%%%%%%%%%%%%%%%%%%%%%%%%%%%%%%%%%%%

\end{document}